\documentclass[11pt,reqno]{amsart}

\numberwithin{equation}{section}
\usepackage[all,cmtip]{xy}
\usepackage{amssymb,rotating,amsmath,amsfonts,amsthm,color,graphics,epsfig,bbm,bm,manfnt}

\newcommand{\calB}{\mathcal{B}}

\newcommand{\calM}{\mathcal{M}}

\newcommand{\calV}{\mathcal{V}}

\newcommand{\mB}{\mathbb{B}}
\newcommand{\mC}{\mathbb{C}}
\newcommand{\mD}{\mathbb{D}}

\newcommand{\mN}{\mathbb{N}}

\newcommand{\mT}{\mathbb{T}}

\newcommand{\mZ}{\mathbb{Z}}

\newcommand{\sH}{{\mathcal{M} ({\textrm{H}}^2_n)}}

\newcommand{\sHH}{{\mathcal{M} ({\textrm{\em H}}^2_n)}}

\newtheorem{theorem}{Theorem}[section]
\newtheorem{lemma}[theorem]{Lemma}

\theoremstyle{definition}

\theoremstyle{definition}
\newtheorem{definition}[theorem]{Definition}

\theoremstyle{definition}

\begin{document}



\title[Noncoherence of $\sH$]{Noncoherence of the multiplier algebra 
of the Drury-Arveson space ${\textrm{H}}^2_n$ for $n\geq 3$}

\author[A. Sasane]{Amol Sasane}

\address{Department of Mathematics \\London School of Economics\\
     Houghton Street\\ London WC2A 2AE\\ United Kingdom}

\email{sasane@lse.ac.uk}

\begin{abstract}
Let ${\textrm{H}}_{n}^2$ denote the Drury-Arveson Hilbert space on the unit ball $\mB_n$ in $\mC^n$, 
and let $\calM({\textrm{H}}_{n}^2)$ be its multiplier algebra. We show that for $n\geq 3$, the ring $\calM({\textrm{H}}_{n}^2)$ is 
not coherent. 
\end{abstract}

\maketitle

\section{Introduction}

\noindent The aim of this article is to investigate a certain algebraic property of rings, 
called {\em coherence}, which is a generalization of the property of being Noetherian, 
for a particular algebra of holomorphic functions in the unit ball in $\mC^n$. 

\begin{definition}[Coherent ring] 
Let $R$ be a unital commutative ring, and for an $n\in \mN:=\{1,2,3,\cdots\}$, 
let $R^n=R\times \cdots \times R$ ($n$ times). If 
$f\in R^n$, say $f=(f_1,\cdots, f_n)$, then a 
{\em relation} $g$ on $f$, written $g\in f^\perp$, is an $n$-tuple 
$g=(g_1,\cdots, g_n)\in R^n$ such that 
 $
g_1 f_1+\cdots + g_n f_n=0.
$ 
The ring $R$ is said to be {\em coherent} if for each $n$ and each 
$f\in R^n$, the $R$-module $f^\perp$ is finitely generated. 
\end{definition}

\noindent A property which is equivalent to coherence  
is that the intersection of any two finitely generated 
ideals in $R$ is finitely generated, and the annihilator 
of any element is finitely generated \cite{Cha}. 
We refer the reader to the monograph \cite{GlaB} for the relevance 
of the property of coherence in homological  algebra. All Noetherian rings 
are coherent, but not all coherent rings are Noetherian. For example, the polynomial ring $\mC[x_1,x_2,x_3,\cdots]$ 
is not Noetherian (because the sequence of ideals 
$
\langle x_1 \rangle\subset \langle x_1, x_2 \rangle\subset \langle x_1, x_2, x_3 \rangle \subset \cdots 
$
is ascending and not stationary), but $\mC[x_1, x_2, x_3,\cdots]$ is coherent \cite[Corollary~2.3.4]{GlaB}. 
 
For algebras of holomorphic functions in the unit disc 
$$
\mD:=\{z\in \mC: |z|<1\}
$$ 
in $\mC$, it is known 
that the Hardy algebra ${\textrm{H}}^\infty(\mD)$, consisting of all bounded and holomorphic 
functions on $\mD$ with pointwise operations, is coherent, 
while the disc algebra $\textrm{A}(\mD)$ (of all functions in $\textrm{H}^\infty(\mD)$ that admit a continuous extension to the 
closure of $\mD$ in $\mC$) is not coherent \cite{McVRub}. 
For $n\geq 3$, Amar \cite{Ama} showed that the 
Hardy algebra ${\textrm{H}}^\infty (\mB_n)$, consisting of all bounded and holomorphic functions in the unit ball 
$$
\mB_n:=\{z=(z_1,\cdots, z_n)\in \mC^n: |z_1|^2+\cdots+|z_n|^2<1\},
$$ 
is not coherent. Related results about some 
other subalgebras of holomorphic functions in the ball 
and the polydisc were also obtained in \cite{Ama}. Whether or not  
the Hardy algebra ${\textrm{H}}^\infty(\mD^2)$ (of the bidisc $\mD^2$) and  
${\textrm{H}}^\infty (\mB_2)$ 
 are coherent does not seem to be known. 
 
The aim of this article is to prove  the noncoherence of the 
multiplier algebra of the Drury-Arveson space in 
$\mC^n$ with $n\geq 3$, and our main result is the following. 

\begin{theorem}
\label{main_theorem}
 For $n\geq 3$, $\sHH$ is not coherent.
\end{theorem}

\noindent We give the pertinent definitions and notation below. 

A multivariable analogue of the classical Hardy space 
on $\mD$ in $\mC$ is the Drury-Arveson space ${\textrm{H}}_{n}^2$ 
on the unit ball $\mB_n$ in $\mC^n$ \cite{Arv}, \cite{Dru}. The space 
${\textrm{H}}_{n}^2$ is a Hilbert function space that has a natural 
$n$-tuple of operators acting on it, giving it the 
structure of a Hilbert module, and has been the object 
of intensive study in the last decade or so owing 
to its relation to multivariable operator theory 
(for example the von Neumann inequality for commuting 
row contractions \cite{Dru}) and multivariable function 
theory (for instance Nevanlinna-Pick interpolation \cite{BalBol}). 

\begin{definition}[The Drury-Arveson space ${\textrm{H}}_{n}^2$] 
The {\em Drury-Arveson space} ${\textrm{H}}_{n}^2$ is a reproducing 
kernel Hilbert space of holomorphic functions on $\mB_n$ 
with the kernel 
$$
K(z,w)=\frac{1}{1-\langle z,w\rangle}, \quad z,w\in \mB_n.
$$  
\end{definition}
 
\noindent We will use the standard multi-index notation: 
For $\alpha=(\alpha_1,\cdots, \alpha_n)\in \mZ_+^n$, 
where $\mZ_+:=\{0,1,2,3,\cdots\}$, 
$$
\alpha!:=\alpha_1 ! \; \alpha_2!\cdots \alpha_n!, \quad |\alpha|:=\alpha_1+\cdots +\alpha_n, \quad 
\zeta^\alpha:= \zeta_1^{\alpha_1} \cdots \zeta_n^{\alpha_n}.
$$
 
\begin{definition}[The multiplier algebra $\calM({\textrm{H}}_{n}^2)$]
  A holomorphic function $f$ on $\mB_n$ is called a {\em multiplier}  for ${\textrm{H}}_{n}^2$ if $f\cdot {\textrm{H}}_{n}^2 \subset {\textrm{H}}_{n}^2$. 
  
   $\calM({\textrm{H}}_{n}^2)$ is the ring of all multipliers on ${\textrm{H}}_{n}^2$ with pointwise operations. 
 \end{definition}
 
\noindent  If $f$ is a multiplier, then the multiplication operator $M_f:{\textrm{H}}_{n}^2 \rightarrow {\textrm{H}}_{n}^2$ corresponding to $f$
 defined by 
 $$
 M_f (g):=fg, \quad g \in {\textrm{H}}_{n}^2,
 $$
is necessarily bounded on ${\textrm{H}}_{n}^2$ \cite{Arv}, and the multiplier norm of $f$ in $\calM({\textrm{H}}_{n}^2)$ is defined to be the 
operator norm of $M_f$. Then $\calM({\textrm{H}}_{n}^2)$ is a strict sub-algebra of 
${\textrm{H}}^\infty(\mB_n)$ if $n\geq 2$ \cite{Arv}. 
If $n=1$, then ${\textrm{H}}_{n}^2={\textrm{H}}_1^2$ is the usual Hardy space of the disc, 
and $\calM({\textrm{H}}_{n}^2)={\textrm{H}}^\infty(\mD)$, the Hardy algebra on 
the disc $\mD$.

The proof of our main result, Theorem~\ref{main_theorem}, 
is an adaption to the case of $\sH$ of the proof given 
in Amar \cite{Ama} for showing the noncoherence of 
${\textrm{H}}^\infty(\mD^n)$, $n\geq 3$.

\section{Preliminaries}

\noindent The following result is shown along the same lines as 
the calculation done in \cite[Lemma~2.3]{FanXia}, 
where it was shown that 
$$
\frac{z_2}{1-s z_1} \in \sH
$$
for all real $s\in (0,1)$. 

\begin{lemma}
\label{norm_multiplier}
Let $\alpha \in \mT:=\{z\in \mC: |z|=1\}$. 
The function $G_\alpha:\mB_n \rightarrow \mC$, given by 
$$
G_\alpha(z)=\frac{z_2}{(1-\alpha z_1^2)^{1/4}}, \quad z=(z_1,\cdots, z_n)\in \mB_n,
$$
belongs to $\calM(H^2_n)$.
\end{lemma}

\noindent Before proving this result, we need some preliminaries from \cite[Section~2]{FanXia}, 
reproduced here for the convenience of the reader as they will play an essential role in 
the justification of Lemma~\ref{norm_multiplier}.
 Let 
 $$
 \calB:=\{(0,\beta_2,\cdots, \beta_n): \beta_2 ,\cdots, \beta_n\in \mZ_+\}\subset \mZ_+^n.
 $$
 We will denote as before, the components of 
 $z$ by $z_1, \cdots, z_n$. For each $\beta \in \calB$, 
 define the closed linear subspace 
 $$
 H_\beta=\overline{\textrm{span}\{z_1^k z^\beta: k\geq 0\}}
 $$
 of ${\textrm{H}}_{n}^2$. Then we have the orthogonal decomposition 
 $$
 {\textrm{H}}_{n}^2 =\displaystyle \bigoplus_{\beta \in \calB} H_\beta.
 $$
 For each $\beta \in \calB$, we have an orthonormal basis $\{e_{k,\beta}: k\geq 0\}$ for $H_\beta$, 
 where 
 \begin{equation}
 \label{eq_prelim_FX_1}
 e_{k,\beta}(z)=\sqrt{\frac{(k+|\beta|)!}{k! \beta!}}z_1^k z^\beta.
\end{equation}
Then $H_0 ={\textrm{H}}_1^2$, the Hardy space of the unit disc $\mD$. 
For the proof of Lemma~\ref{norm_multiplier}, we need to 
identify each $H_\beta$, $\beta \neq 0$, as a weighted Bergman space on the unit disc. 

Let $dA$ be the area measure on $\mD$ with the normalization $ A(\mD)=1$. 
For each integer $m\geq 0$, let 
$$
\textrm{B}^{(m)}:= \textrm{L}_{a}^2 \Big(\mD, (1-|\zeta|^2)^m  dA(\zeta)\Big),
$$
the usual weighted Bergman space of weight $m$. Then 
$$
\{e_k^{(m)}:k\in \mZ_+\}
$$
is the standard orthonormal basis for $\textrm{B}^{(m)}$, where 
\begin{equation}
 \label{eq_prelim_FX_2}
e_k^{(m)}(\zeta)=\sqrt{\frac{(k+m+1)!}{k!m!}}\zeta^k.
\end{equation}
For each $\beta \in \calB \setminus \{0\}$, define the unitary operator $W_\beta:H_\beta \rightarrow \textrm{B}^{(|\beta|-1)}$ by 
\begin{equation}
 \label{eq_prelim_FX_3}
W_\beta e_{k,\beta}=e_k^{(|\beta|-1)},\quad  k\in \mZ_+.
\end{equation}
It follows from \eqref{eq_prelim_FX_1} and \eqref{eq_prelim_FX_2} that the weighted shift $M_{z_1}|H_\beta$ 
is unitarily equivalent to $M_\zeta$ on $\textrm{B}^{(|\beta|-1)}$. Thus if $\beta \in \calB\setminus \{0\}$, then 
$$
W_\beta M_{z_1} h_\beta = M_\zeta W_\beta h_\beta\;\; \textrm{ for all } h \in H_\beta.
$$
Note that $M_{z_1}|H_0$ is the unilateral shift. 

We will also need the following fact.

\begin{lemma}
\label{lemma_Carleson}
  $|1-\zeta^2|^{-1/2} dA(\zeta)$ is a Carleson measure for the Hardy space ${\textrm{\em H}}_1^2$ of the unit disc $\mD$. 
\end{lemma}
\begin{proof} For $z=e^{i\varphi}$, where $\varphi \in (-\pi,\pi]$, let 
 $$
S_\theta(z):=\{re^{it}: 1-\theta \leq r < 1, \; |t- \varphi|\leq  \theta \}.
$$ 
Then we have 
\begin{eqnarray*}
 \iint_{S_\theta(z)} |1-\zeta^2|^{-1/2} dA(\zeta) 
 &=& 
 \int_{\varphi-\theta}^{\varphi+\theta} \int_{1-\theta}^1 \frac{1}{|1-(re^{it})^2|^{1/2}} r dr dt \\
 &=& 
  \int_{\varphi-\theta}^{\varphi+\theta} \int_{1-\theta}^1 \frac{1}{\sqrt[4]{1-2r^2 \cos (2t)+r^4}} r dr dt \\
  &\leq& 
  \int_{\varphi-\theta}^{\varphi+\theta} \int_{1-\theta}^1 \frac{1}{\sqrt[4]{1-2r^2 +r^4}} r dr dt \\
& =&\int_{\varphi-\theta}^{\varphi+\theta} \int_{1-\theta}^1 \frac{1}{\sqrt{1-r^2 }} r dr dt\\
 &=& 
   \int_{\varphi-\theta}^{\varphi+\theta} \int_0^{1-(1-\theta)^2} \frac{1}{2\sqrt{u}} du  dt \textrm{ (with }u=1-r^2\textrm{)}
   \\
&=& 
   \int_{\varphi-\theta}^{\varphi+\theta} \sqrt{u} \Big|_{0}^{1-(1-\theta)^2} dt 
\\
&=&\int_{\varphi-\theta}^{\varphi+\theta}\sqrt{1-(1-\theta)^2} dt
 \\
  &\leq & 
  \int_{\varphi-\theta}^{\varphi+\theta} 1 dt = 2\theta.
\end{eqnarray*}
This completes the proof. 
\end{proof}

We are now ready to prove Lemma~\ref{norm_multiplier}. 

\begin{proof}[Proof of Lemma~\ref{norm_multiplier}] It is enough to consider the case when $\alpha=1$. 
Let $h_\beta \in H_\beta$, where 
$\beta=(0,\beta_2,\cdots, \beta_n)$. Then 
$$
h_\beta(z)= \sum_{k=0}^\infty c_k z_1^k z^\beta.
$$
First we assume that $\beta\neq 0$. By \eqref{eq_prelim_FX_3}, 
$$
(W_\beta h_\beta)(\zeta)=\sqrt{\frac{\beta !}{(|\beta|-1)!}}\sum_{k=0}^\infty c_k \zeta^k , \quad \zeta \in \mD.
$$
Then $W_\beta h_\beta \in {\textrm{B}}^{(|\beta|-1)}$. Denote $e_2=(0,1,\cdots,0)$. Since $z_2 z^\beta =z^{\beta+e_2}$, we have 
$$
(W_{\beta+e_2} z_2 h_\beta)(\zeta)= \sqrt{\frac{(\beta+e_2)!}{|\beta|!}}\sum_{k=0}^\infty c_k \zeta^k , \quad \zeta \in \mD,
$$
and $W_{\beta+e_2} z_2 h_\beta \in {\textrm{B}}^{|\beta|}$. Now suppose that 
$$
h_\beta(z)=(1-z_1^2)^{-1/4} f_\beta (z),
$$
where 
$$
f_\beta(z)=\sum_{k=0}^\infty a_k z_1^k z^\beta.
$$
For $\zeta \in \mD$, we have $|1-\zeta^2|\geq 1-|\zeta|^2$, and so 
\begin{equation}
\label{eq_norm_multiplier_a}
|1-\zeta^2|^{1/2}\geq (1-|\zeta|^2)^{1/2}\geq 1-|\zeta|^2.
\end{equation}
We have 
\begin{eqnarray*}
 \|z_2 &&\!\!\!\!\!\!\!\!\!\!\!\!\!\!\!(1-z_1^2)^{-1/4} f_\beta\|_{{\textrm{H}}_{n}^2}^2\phantom{\int_{\mD} }
 \\
 &=&
 \|z_2 h_\beta\|_{{\textrm{H}}_{n}^2}^2=\|W_{\beta+e_2} \zeta_2 h_\beta\|^2_{{\textrm{B}}^{(|\beta|)}}\phantom{\int_{\mD}\sum_{k=0}^\infty }
 \\
 &=& \frac{(\beta+e_2)!}{|\beta|!} \int_{\mD} \Big| \sum_{k=0}^\infty c_k \zeta^k\Big|^2 (1-|\zeta|^2)^{|\beta|} dA(\zeta)\\
 &=& \frac{(\beta+e_2)!}{|\beta|!} \int_{\mD} \Big|\frac{1}{(1-\zeta^2)^{1/4}} \sum_{k=0}^\infty a_k \zeta^k\Big|^2 (1-|\zeta|^2)^{|\beta|} dA(\zeta)\\
&=& \frac{(\beta+e_2)!}{|\beta|!} \int_{\mD} \Big|\sum_{k=0}^\infty a_k \zeta^k\Big|^2 
\frac{(1-|\zeta|^2)^{|\beta|}}{|1-\zeta^2|^{1/2}} dA(\zeta)\\
&\leq & 
\frac{(\beta+e_2)!}{|\beta|!} \int_{\mD} \Big|\sum_{k=0}^\infty a_k \zeta^k\Big|^2 
(1-|\zeta|^2)^{|\beta|-1} dA(\zeta) \textrm{ (using  \eqref{eq_norm_multiplier_a})} \\
 &= &\frac{\beta_2+1}{|\beta|}\frac{\beta!}{(|\beta|-1)!} \int_{\mD} \Big|\sum_{k=0}^\infty a_k \zeta^k\Big|^2 
(1-|\zeta|^2)^{|\beta|-1} dA(\zeta)\\
&=& \frac{\beta_2+1}{|\beta|} \|W_\beta f_\beta\|^2_{{\textrm{B}}^{(|\beta|-1)}}=
\frac{\beta_2+1}{|\beta|}\|f_\beta\|_{{\textrm{H}}_{n}^2}^2 \leq 2\|f_\beta\|_{{\textrm{H}}_{n}^2}^2. \phantom{\int_{\mD}\sum_{k=0}^\infty }
\end{eqnarray*}
So we have shown that for $\beta \neq 0$, the norm of the restriction of the operator of 
multiplication by $z_2 (1-z_1^2)^{-1/4}$ to $H_\beta$ does not exceed $\sqrt{2}$. 

Next we consider the case when $\beta=0$. We know that $H_0={\textrm{H}}_1^2$, the Hardy space on $\mD$. Let 
$h\in H_0$. Then 
$$
h(z)= \sum_{k=0}^\infty c_k z_1^k.
$$
we have 
$$
(W_{e_2} z_2 h)(\zeta)=\sum_{k=0}^\infty c_k \zeta^k, \quad \zeta \in \mD
$$
and $W_{e_2} z_2 h$ belongs to the Bergman space ${\textrm{B}}^{(0)}$. Now suppose 
$$
h(z)=(1-z_1^2)^{-1/4} f(z),
$$
for some 
$$
f(z)= \sum_{k=0}^\infty a_k z_1^k.
$$
Then 
\begin{eqnarray*}
 \|z_2 (1-z_1^2)^{-1/4} f\|_{{\textrm{H}}_{n}^2}^2 
 &=& \|W_{e_2} z_2 h\|_{{\textrm{B}}^{(0)}}^2 \phantom{\displaystyle \int_\mD \Big| \sum_{k=0}^\infty}\\
 &=&\int_\mD \Big| \sum_{k=0}^\infty c_k \zeta^k\Big|^2 dA(\zeta)\\
 &=& \int_\mD \Big| \frac{1}{(1-\zeta^2)^{1/4}} \sum_{k=0}^\infty a_k \zeta^k \Big|^2 dA(\zeta)\\
 &=& \int_{\mD} \Big| \sum_{k=0}^\infty a_k \zeta^k \Big|^2 |1-\zeta^2|^{-1/2} dA(\zeta)\\
 &\leq& C\|f\|_{{\textrm{H}}_1^2}^2,\phantom{ \sum_{k=0}^\infty \int_\mD\frac{1}{1-\zeta^2)^{1/4}}}
\end{eqnarray*}
where the last inequality follows from the fact that $|1-\zeta^2|^{-1/2} dA(\zeta)$ is a Carleson measure for ${\textrm{H}}_1^2$ 
(Lemma~\ref{lemma_Carleson} above). 
 So we have shown that the norm of the restriction of the operator of multiplication by $z_2(1-z_1^2)^{-1/4}$ to $H_0$ 
 does not exceed $\sqrt{C}$. 
 
 If $\beta \neq \beta'$, $f_\beta \in H_\beta$, and $f_{\beta'} \in H_{\beta'}$, then 
 $$
 \frac{z_2}{(1-z_1^2)^{1/4}} f_\beta \;\;\perp\;\; \frac{z_2}{(1-z_1^2)^{1/4}} f_{\beta'}.
 $$
 Thus it follows from the two paragraphs above that the multiplication operator $M_{G_\alpha}$ 
 corresponding to 
 $$
 G_\alpha =\frac{z_2}{(1-z_1^2)^{1/4}}
 $$
 is a continuous linear map on ${\textrm{H}}_{n}^2$, that is, $G_\alpha \in \calM({\textrm{H}}_{n}^2)$. This completes the proof.  
\end{proof}

\section{Noncoherence of $\sH$}

\begin{proof}[Proof of Theorem~\ref{main_theorem}] 
We will prove the claim by contradiction. Suppose that $\sH$ is a 
coherent ring. Let $f=(f_1,f_2) \in (\sH)^2$, where 
$f_1:= z_1$ and $ f_2:= z_2$. As $\sH$ is coherent, $f^\perp$ 
will be finitely generated, say by $h_1,\cdots, h_k$ in $(\sH)^2$.  
For $\alpha \in \mT$, define $g_\alpha=(g_{1,\alpha}, g_{2,\alpha})$ by  
\begin{eqnarray*}
g_{1,\alpha}(z)
&:=&
\displaystyle \frac{z_2}{( 1- \alpha z_3^2)^{1/4}},\quad 
\\
g_{2,\alpha}(z)&:=&
\displaystyle 
\frac{-z_1}{( 1- \alpha z_3^2)^{1/4}},
\end{eqnarray*}
for $z=(z_1,\cdots, z_n)\in \mB_n$. Note that by Lemma~\ref{norm_multiplier}, 
we know that $g_{\alpha} $ is in $\sH$ for each $\alpha \in \mT$. 

The rest of the proof is the same, mutatis mutandis, as the proof given in \cite[Section~1, pages 69-71]{Ama}. 
We repeat it here making sure that the implicit but straightforward changes needed 
in that proof to adapt it to our different situation, are made explicit here for the 
convenience of the reader. 

Moreover,  
$$
f_1 g_{\alpha,1}+ f_2 g_{\alpha,2}= z_1 \cdot \frac{z_2}{( 1- \alpha z_3^2)^{1/4}}
+z_2\cdot \frac{-z_1}{( 1- \alpha z_3^2)^{1/4}}=0,
$$ 
and so $g_\alpha=(g_{1,\alpha}, g_{2,\alpha}) \in f^\perp$. Thus there exist 
$\gamma_{\alpha,i} \in \sH$ such that 
\begin{equation}\label{eq_main_theorem_1}
g_\alpha = \sum_{i=1}^k \gamma_{\alpha,i} h_i.
\end{equation}
If $h_i=:(r_i,s_i)\in (\sH)^2$, then we have 
 $$
z_1 r_i +z_2 s_i =0.
$$ 
So if $z_2=0$, then $ z_1 r_i=0$. Thus $r_i=0$ on 
$$
\{z =(z_1,\cdots, z_n)\in \mB_n:z_2=0\}.
$$
%Similarly, $s_i=0$ on $\{z =(z_1,\cdots, z_n)\in \mB_n:z_1=0\}$. 
Hence there exist $t_i$, holomorphic in $\mB_n$ such that 
 $$
r_i(z)= z_2  t_i(z),\quad  i=1,\cdots, k, \quad z\in \mB_n.
$$ 
So it now follows from \eqref{eq_main_theorem_1} that 
$$
\frac{z_2}{( 1- \alpha z_3^2 )^{1/4}}
 =
  \sum_{i=1}^k \gamma_{\alpha,i}(z) z_2 t_i(z),
 $$
 that is, 
$$
\varepsilon_\alpha (z):= \frac{1}{( 1- \alpha z_3^2 )^{1/4}}
=
\sum_{i=1}^k \gamma_{\alpha,i}(z) t_i(z),\quad \alpha \in \mT,\; z\in \mB_n.
$$
Let $\alpha_1,\cdots, \alpha_k,\alpha_*$ be $k+1$ distinct points on $\mT$. We interpret 
\begin{equation}
\label{eq_Amar_1}
 \varepsilon_\alpha (z)=\sum_{i=1}^k \gamma_{\alpha,i}(z) t_i(z)
\end{equation}
for these $k+1$ choices of $\alpha$ as a system of $k+1$ linear equations in $k$ unknowns, the $t_i(z)$'s:
\begin{equation}
\label{eq_Amar_2}
 \left[\begin{array}{cccc}
 \gamma_{\alpha_1 ,1} & \cdots & \gamma_{\alpha_1, k} &  \varepsilon_{\alpha_1}\\
 \vdots &&\vdots&\vdots \\
 \gamma_{\alpha_k ,1} & \cdots & \gamma_{\alpha_k ,k} &  \varepsilon_{\alpha_k}\\
 \gamma_{\alpha_* , 1} & \cdots & \gamma_{\alpha_* , k} &  \varepsilon_{\alpha_*}
 \end{array}\right]
\underbrace{\left[\begin{array}{cccc}
t_1 \\
\vdots \\
t_k \\
-1
\end{array}\right]}_{\neq 0}
=
0.
\end{equation}
Since \eqref{eq_Amar_2} is solvable, we must have 
$$
\det \left[\begin{array}{cccc}
 \gamma_{\alpha_1 ,1} & \cdots & \gamma_{\alpha_1 ,k} &  \varepsilon_{\alpha_1}\\
 \vdots &&\vdots&\vdots\\
 \gamma_{\alpha_k , 1} & \cdots & \gamma_{\alpha_k ,k} &  \varepsilon_{\alpha_k}\\
 \gamma_{\alpha_* , 1} & \cdots & \gamma_{\alpha_* , k} &  \varepsilon_{\alpha_*}
 \end{array}\right]=0.
$$
Expanding the determinant along the last column gives 
\begin{equation}
\label{eq_Amar_3}
\underbrace{\det \left[\begin{array}{ccc}
\gamma_{\alpha_1 ,1} & \cdots & \gamma_{\alpha_1, k}\\
\vdots && \vdots\\
\gamma_{\alpha_k ,1} & \cdots & \gamma_{\alpha_k ,k}
\end{array}\right]}_{=:\Delta}
\cdot\; \varepsilon_{\alpha_*}
=
\sum_{i=1}^k \Lambda_{\alpha_* , i} \cdot \varepsilon_{\alpha_i},
\end{equation}
with $\Lambda_{\alpha_* ,i}\in \sH\subset {\textrm{H}}^\infty (\mB_n)$ (since the $\gamma_{\alpha_j ,i}\in \sH$). 
Now we consider the following two possible cases separately:
\begin{itemize}
 \item[$\underline{1}^\circ$] The determinant $\Delta$ is not identically $0$ on the variety

  \noindent $\calV:=\{z=(z_1,\cdots, z_n)\in \mB_n: z_1=z_2=0\}$.
 \item[$\underline{2}^\circ$] $\Delta\equiv 0$ on $\calV$. 
\end{itemize}
Let us consider case $1^\circ$ first. The map $z_3\mapsto \Delta|_{\calV}(0,0,z_3):\mD \rightarrow \mC$ 
is holomorphic and bounded, independent of the $\alpha_*$. As $\Delta|_{\calV}$ is not identically 
zero, there exists a point $\alpha_*\in \mT$, which is distinct from $\alpha_1,\cdots,\alpha_k$, 
such the radial limit of $\Delta|_{\calV}(0,0,\cdot)$ is nonzero 
as $z_3\rightarrow {\overline{\alpha_*}}^{1/2}$. Then  $z_3^2$ approaches $\overline{\alpha_*}$, 
and we see in \eqref{eq_Amar_3} that the left hand side approaches $\infty$, 
while it is not the case that the right hand side approaches $\infty$ 
(because the $\Lambda^i_{\alpha_*}$ and the $\varepsilon_{\alpha_j}$, 
with $\alpha_j\neq\alpha_*$, stay bounded). This contradiction shows that this case can't 
be possible. 

So we now consider case $2^\circ$. Suppose that $\Delta=0$ on $\calV$ 
for every choice of $\alpha_1,\cdots,\alpha_k$ in $\mT$. Let $\ell$ be the rank 
$$
\ell:=
\textrm{rank}_{\calV}\left[\begin{array}{ccc}
\gamma_{\alpha_1, 1} & \cdots & \gamma_{\alpha_1, k}\\
\vdots && \vdots\\
\gamma_{\alpha_k ,1} & \cdots & \gamma_{\alpha_k, k}
\end{array}\right]
:= \max_{z\in \calV} \textrm{rank}\left[\begin{array}{ccc}
\gamma_{\alpha_1, 1}(z) & \cdots & \gamma_{\alpha_1, k}(z)\\
\vdots && \vdots\\
\gamma_{\alpha_k, 1}(z) & \cdots & \gamma_{\alpha_k, k}(z)
\end{array}\right].
$$
Thus $\ell<k$ owing to the fact that $\Delta=0$ on $\calV$. 
After a  rearrangement (if necessary) of the $\alpha_i$, we arrive at 
$$
\det 
\left[\begin{array}{ccc}
\gamma_{\alpha_1, 1} & \cdots & \gamma_{\alpha_1, \ell}\\
\vdots && \vdots\\
\gamma_{\alpha_\ell, 1} & \cdots & \gamma_{\alpha_\ell, \ell}
\end{array}\right]
\not\equiv 0 \textrm{ on } \calV.
$$
From \eqref{eq_Amar_1}, we can deduce that $\ell$ can't be zero. 
Indeed, otherwise all the $\gamma_{\alpha_j, i}\equiv 0$ on 
$\calV$ and by \eqref{eq_Amar_1}, we would have $1/(1-\alpha z^2)^{1/4}=0$, $z\in \mD$, 
which is clearly impossible. So we have that $\ell\geq 1$, and from the definition of the rank 
it follows that 
$$
D_{ij}= \det \left[\begin{array}{cccc}
\gamma_{\alpha_1, 1} & \cdots & \gamma_{\alpha_1, \ell} & \gamma_{\alpha_1, i}\\
\vdots && \vdots& \vdots \\
\gamma_{\alpha_\ell, 1} & \cdots & \gamma_{\alpha_\ell, \ell} & \gamma_{\alpha_\ell, i}\\
\gamma_{\alpha_j, 1} & \cdots & \gamma_{\alpha_j, \ell} & \gamma_{\alpha_j, i}
\end{array}\right]
\equiv 0 
\textrm{ on }\calV \textrm{ for all }i,j \textrm{ in }\{1,\cdots,k\}.
$$
We have 
\begin{eqnarray*}
\det &&\!\!\!\!\!\!\!\!\!\!\!\!\!\!\left[\begin{array}{cccc}
\gamma_{\alpha_1, 1} & \cdots & \gamma_{\alpha_1, \ell} & \varepsilon_{\alpha_1}\\
\vdots && \vdots& \vdots \\
\gamma_{\alpha_\ell, 1} & \cdots & \gamma_{\alpha_\ell, \ell} & \varepsilon_{\alpha_\ell}\\
\gamma_{\alpha_j, 1} & \cdots & \gamma_{\alpha_j, \ell} & \varepsilon_{\alpha_j}
\end{array}\right]
\\
&=&
\det \left[\begin{array}{c|c}\!\!\!
   \begin{array}{ccc}
   \gamma_{\alpha_1, 1} & \cdots\!\!\!\!\!\! & \gamma_{\alpha_1, \ell} \\
   \vdots & & \vdots \\
   \gamma_{\alpha_\ell, 1} & \cdots\!\!\!\!\!\! & \;\;\; \;\;\gamma_{\alpha_\ell, \ell}\phantom{\displaystyle \sum}
   \end{array} \!\!\!\!\!\!\!\!\!
   & 
   t_1 \left[\!\!\! \begin{array}{ccc} \gamma_{\alpha_1, 1}\\ \vdots \\ \gamma_{\alpha_\ell,1}\end{array}\!\!\!\right]
   + \cdots +
   t_k \left[\!\!\!\begin{array}{ccc} \gamma_{\alpha_1,k}\\ \vdots \\ \gamma_{\alpha_\ell, k}\end{array}\!\!\!\right]
   \\ \hline 
   \begin{array}{ccc} \gamma_{\alpha_j, 1}\;\; & \!\!\!\!\!\cdots & \gamma_{\alpha_j, \ell} \end{array} 
   & 
   t_1 \gamma_{\alpha_j, 1}+ \cdots + t_k \alpha_{\alpha_j, k}
\end{array}\right]
\\
&=&\sum_{i=1}^k t_i D_{ij}\equiv 0\textrm{ on }\calV\textrm{ for all }j\in \{1,\cdots,k\}. 
\end{eqnarray*}
By expanding the determinant on the left hand side along the last column, we obtain
$$
\underbrace{\det \left[\begin{array}{ccc}
\gamma_{\alpha_1, 1} & \cdots & \gamma_{\alpha_1, \ell}\\
\vdots && \vdots\\
\gamma_{\alpha_\ell, 1} & \cdots & \gamma_{\alpha_\ell, \ell}
\end{array}\right]}_{=:\delta}
\cdot \;\varepsilon_{\alpha_j}
=
\sum_{i=1}^\ell \lambda_{\alpha_j, i} \cdot \varepsilon_{\alpha_i}\textrm{ on }\calV,
$$
with $\lambda_{\alpha_* , i}\in \sH \subset {\textrm{H}}^\infty(\mB_n)$. If it is not the case that 
$\delta\equiv 0$ on $\calV$, then we repeat the argument in $1^\circ$ (replacing $\alpha_*$ by $\alpha_j$), 
and arrive at a contradiction. So we conclude that $\delta\equiv 0$ on $\calV$, but this contradicts the definition of the 
 $\ell$. Hence case $2^\circ$ is impossible too. 

 Consequently, $f^\perp$ is not finitely generated, and so $\sH$ is not coherent. 
\end{proof}

\smallskip 

\noindent {\bf Acknowledgement:}  I would like to thank
Professor Jingbo Xia (State University of New York at Buffalo)
for showing me an outline of the proof of Lemma~\ref{norm_multiplier} 
and for several useful discussions relating to it.

\end{document}